\documentclass{amsart}
\usepackage{latexsym}
\usepackage{amsmath}
\usepackage{amssymb}
\usepackage{amsfonts}

\numberwithin{equation}{section}
\newcommand{\dist}{\mathop{\mathrm{dist}}}

\newtheorem{theorem}{Theorem}[section]
\newtheorem{lemma}[theorem]{Lemma}
\newtheorem{remark}[theorem]{Remark}
\newtheorem{proposition}[theorem]{Proposition}

\newcommand{\supp}{\mbox{\rm supp\,}}

\begin{document}

\title[Increasing stability of IBVP]{Increasing stability of the inverse boundary value problem for the Schr\"odinger equation}

\author{V Isakov}
\address{Department of Mathematics and Statistics, Wichita State University, KS 67260-0033, USA.}
\email{victor.isakov@wichita.edu}
\thanks{Isakov was in part supported by the NSF grant DMS 10-08902 and by Emylou Keith and Betty Dutcher Distinguished Professorship at WSU}

\author{S Nagayasu}
\address{Department of Mathematical Sciences, University of Hyogo, Himeji 671-2201,  Japan.}
\email{sei@sci.u-hyogo.ac.jp}
\thanks{Nagayasu was partly supported by Grant-in-Aid for Young Scientists (B)}

\author{G Uhlmann}
\address{Department of Mathematics, University of Washington, Box 354305, Seattle, WA 98195-4350, USA \& Fondation de Sciences Math\'ematiques de Paris}
\email{gunther@math.washington.edu}
\thanks{Uhlmann was partly supported by NSF and The Fondation de Sciences Math\'ematiques de Paris}

\author{J-N Wang}
\address{Department of Mathematics, NCTS (Taipei), National Taiwan University, Taipei 106, Taiwan.}
\email{jnwang@math.ntu.edu.tw}
\thanks{Wang was partly supported by the NSC grant 99-2115-M-002-006-MY3.}




\begin{abstract}
In this work we study the phenomenon of increasing stability in the inverse boundary value problem for the Schr\"odinger equation. This problem was previously considered by Isakov in which he discussed the phenomenon in different ranges of the wave number (or energy). The main contribution of this work is to provide a unified and easier approach to the same problem based on the complex geometrical optics solutions.
\end{abstract}

\maketitle

\section{Introduction}\label{sec1}

Most of inverse problems are known to be severely ill-posed. This weakness makes it extremely difficult to design reliable reconstruction algorithms in practice. However, in some cases, it has been observed numerically that the stability increases with respect to some parameter such as the wave number (or energy) (see, for example, \cite{CHP} for the inverse obstacle scattering problem).  Several rigorous justifications of the increasing stability phenomena in different
settings were obtained by Isakov {\it et al} \cite{HI, I07, I11, ASI07, ASI10}. In particular, in \cite{I11}, Isakov considered the Helmholtz equation with a potential
\begin{equation}\label{eq:1}
\bigl( \Delta + k^{2} + q(x) \bigr) u(x) = 0 \mbox { in }
\Omega \subset \mathbb{R}^{n}
\end{equation}
with $n\ge 3$. He obtained stability estimates of determining $q$ by the
Dirichlet-to-Neumann map for different ranges of $k$, which demonstrate the increasing stability phenomena in $k$. The purpose of this work is to provide a more straightforward way to derive a similar estimate for the inverse boundary value for \eqref{eq:1}. In \cite{I11}, Isakov used \emph{real} geometrical optics solutions for the large wave number $k$. In this work, by more careful choice of an additional large parameter and a priori constraints we are able to use \emph{complex} geometrical optics (CGO) solutions introduced by Calder\'on \cite{ca80} and Sylvester-Uhlmann \cite{su87} for all  $k\geq 1$. This will simplify the proof in \cite{I11}. Recently similar results were obtained by Isaev and Novikov \cite{IN} by using less explicit and more complicated methods of scattering theory.

In this work, instead of considering the Dirichlet-to-Neumann map, we define the boundary measurements to be the Cauchy data corresponding to \eqref{eq:1}
\[
\mathcal{C}_{q} = \left\{
 \left(
  u|_{\partial\Omega},
  \frac{\partial u}{\partial\nu} \biggr|_{\partial\Omega}
 \right) , \  \text{where} \ u \
 \text{is a solution to (\ref{eq:1})}
\right\}.
\]
Hereafter, $\nu$ is the unit outer normal vector of $\partial
\Omega$. Assume that $\mathcal{C}_{q_1}$ and
$\mathcal{C}_{q_2}$ are two
Cauchy data associated with refraction indices $q_1$ and $q_2$,
respectively. To measure the distance between two Cauchy data, we
define
\begin{align*}
\dist ( \mathcal{C}_{q_1}, \mathcal{C}_{q_2} )
& = \max \left\{
 \max_{(f,g) \in \mathcal{C}_{q_{1}}}
 \min_{( \widetilde{f}, \widetilde{g} ) \in \mathcal{C}_{q_{2}}}
 \frac{\lVert (f,g) - ( \widetilde{f}, \widetilde{g} ) \rVert_{
  H^{1/2} \oplus H^{-1/2}
 }}{\lVert (f,g) \rVert_{H^{1/2} \oplus H^{-1/2}}},\right.\\
& \qquad \qquad \ \left.
 \max_{(f,g) \in \mathcal{C}_{q_{2}}}
 \min_{( \widetilde{f}, \widetilde{g} ) \in \mathcal{C}_{q_{1}}}
 \frac{\lVert (f,g) - ( \widetilde{f} , \widetilde{g} ) \rVert
  _{H^{1/2} \oplus H^{-1/2}
 }}{\lVert (f,g) \rVert_{H^{1/2} \oplus H^{-1/2}}}
\right\} ,
\end{align*}
where
\[
\lVert (f,g) \rVert_{H^{1/2} \oplus H^{-1/2}}
=\bigl(
 \lVert f \rVert_{H^{1/2} ( \partial \Omega )}^2
 +\lVert g \rVert_{H^{-1/2} ( \partial \Omega) }^2
\bigr)^{1/2}.
\]
Our main theorem is stated as follows.
\begin{theorem}\label{theorem:main}
Let $n \geq 3$.
Assume $\mathcal{C}_{q_{1}}$ and $\mathcal{C}_{q_{2}}$
are Cauchy data corresponding to $q_{1} (x)$ and $q_{2} (x)$,
respectively.
Let $s > n/2$ and $M > 0$.
Assume $\lVert q_{l} \rVert_{H^{s} ( \Omega )} \leq M$
{\rm (}$l = 1, 2${\rm )} and
$\supp ( q_{1} - q_{2} ) \subset \Omega$.
Denote $\widetilde{q}$ the zero extension of $q_{1} - q_{2}$.
Then for $k \geq 1$ and
$\dist ( \mathcal{C}_{q_{1}} , \mathcal{C}_{q_{2}} ) \leq 1 / e$
we have the following stability estimate:
\begin{equation}\label{eq:main}
\lVert \widetilde{q} \rVert_{H^{-s} ( \mathbb{R}^{n} )}
\leq C k^{4} \dist ( \mathcal{C}_{q_{1}} , \mathcal{C}_{q_{2}} )
+ C \left(
 k + \log \frac{1}{
  \dist ( \mathcal{C}_{q_{1}} , \mathcal{C}_{q_{2}} )
 }
\right)^{- (2s-n)} ,
\end{equation}
where $C > 0$ depends only on $n, s, \Omega , M$
and $\supp ( q_{1} - q_{2} )$.
\end{theorem}

From estimate \eqref{eq:main}, it is obvious that the stability behaves more like Lipschitz type when $k$ is large. We would like to point out that unlike in the acoustic case where the constant associated with the Lipschitz estimate grows exponentially in $k$ \cite{NUW}, the constant here grows only polynomially in $k$. Similarly, the corresponding constant obtained in \cite{I11} (see estimate (8) there) also grows polynomially in $k$.

The paper is organized as follows. In Section~2, we will collect some known results about the CGO solutions and an estimate for the difference of potentials, which are essential tools in the proof. In Section~3, we present a detailed proof of Theorem~\ref{theorem:main}.

\section{Preliminaries}

To begin, we state the existence of CGO solutions for \eqref{eq:1}. These special solutions are first constructed by Sylvester and Uhlmann \cite{su87}. Another construction based on the Fourier series is given by H\"ahner \cite{ha}.
\begin{lemma}\label{lemma:CGOsolution}
Let $s > n/2$.
Assume that $\zeta = \eta + i \xi$
{\rm (}$\eta , \xi \in \mathbb{R}^{n}${\rm )} satisfies
\[
\lvert \eta \rvert^{2} = k^{2} + \lvert \xi \rvert^{2} \quad
\mbox{and} \quad
\eta \cdot \xi = 0,
\]
i.e., $\zeta\cdot\zeta=k^2$. Then there exist constants $C_{\ast}$
and $C > 0$,
which are independent of $k$, such that
if $\lvert \xi \rvert > C_{\ast} \lVert q \rVert_{H^{s} ( \Omega )}$
then there exists a solution $u$ to the equation
{\rm (\ref{eq:1})}
of the form
\begin{equation}\label{cgo}
u(x) = e^{i \zeta \cdot x} \bigl( 1 + \psi (x) \bigr) ,
\end{equation}
where $\psi$ has the estimate
\begin{equation*}
\lVert \psi \rVert_{H^{s} ( \Omega )}
\leq \frac{C}{\lvert \xi \rvert}
\lVert q \rVert_{H^{s} ( \Omega )} .
\end{equation*}
\end{lemma}
\begin{remark}
Note that the correction term $\psi$ decays in  $\text{Im}\,\zeta$. This property is crucial in obtaining that the constant associated with the Lipschiz estimate grows only polynomially in $k$.
\end{remark}

Next inequality is an easy consequence of Alessandrini's identity. We refer to  \cite{fu} for the proof.
\begin{proposition}\label{proposition:identityestimate}
Let $u_{l}$ and $\mathcal{C}_{q_{l}}$ be
solution and Cauchy data to the equation {\rm (\ref{eq:1})}
with $q = q_{l}$, respectively {\rm (}$l = 1, 2${\rm )}.
Then the following estimate holds:
\begin{align*}
& \left\lvert
 \int_{\Omega} ( q_{2} - q_{1} ) u_{1} u_{2} \, d x
\right\rvert \\
& \leq \left\lVert
 \left( u_{1} , \frac{\partial u_{1}}{\partial \nu} \right)
\right\rVert_{H^{1/2} \oplus H^{-1/2}}
\left\lVert
 \left( u_{2} , \frac{\partial u_{2}}{\partial \nu} \right)
\right\rVert_{H^{1/2} \oplus H^{-1/2}}
\dist ( \mathcal{C}_{q_{1}} , \mathcal{C}_{q_{1}} ).
\end{align*}
\end{proposition}

\section{Proof of main theorem}

To prove Theorem~\ref{theorem:main}, we first derive two lemmas.
\begin{lemma}\label{lemma:keyestimate}
Under the assumptions in Theorem~\ref{theorem:main},
\begin{equation}\label{eq:lemma01}
\lvert \mathcal{F} \widetilde{q} ( r \omega ) \rvert
\leq C k^{4} e^{C a}
\dist ( \mathcal{C}_{q_{1}} , \mathcal{C}_{q_{2}} )
+ \frac{C}{a} \lVert \widetilde{q} \rVert_{H^{-s} ( \mathbb{R}^{n} )}
\end{equation}
holds for $k \geq 1$,
$r \geq 0$, $\omega \in \mathbb{R}^{n}$
with $\lvert \omega \rvert = 1$
and $a > C_{\ast} M$ with
$k^{2} + a^{2} > r^{2} / 4$,
where $C > 0$ depends only on
$n, s, M, \Omega$ and $\supp ( q_{1} - q_{2} )$
and $C_{\ast}$ is the constant given in Lemma~\ref{lemma:CGOsolution}.
\end{lemma}
\begin{proof}

We will use CGO solutions \eqref{cgo} with appropriately chosen parameter $\zeta$.  Let us denote $\zeta_{l} = \eta_{l} + i \xi_{l}$, $l = 1, 2$. We can choose
$\omega^{\bot}, \widetilde{\omega}^{\bot} \in \mathbb{R}^{n}$
satisfying
\[
\omega \cdot \omega^{\bot}
= \omega \cdot \widetilde{\omega}^{\bot}
= \omega^{\bot} \cdot \widetilde{\omega}^{\bot} = 0
\quad \mbox{and} \quad
\lvert \omega^{\bot} \rvert
= \lvert \widetilde{\omega}^{\bot} \rvert = 1.
\]
Now we set
\begin{align*}
& \xi_{1} = a \omega^{\bot} , \quad
\eta_{1} = - \frac{r}{2} \omega
+ \sqrt{k^{2} + a^{2} - \frac{r^{2}}{4}} \,
\widetilde{\omega}^{\bot} , \\
& \xi_{2} = - \xi_{1} \quad \mbox{and} \quad
\eta_{2} = - r \omega - \eta_{1},
\end{align*}
and thus
\[
\xi_{l} \cdot \eta_{l}  = 0, \quad
\lvert \eta_{l} \rvert^{2} = k^{2} + \lvert \xi_{l} \rvert^{2}
\]
and
\begin{math}
\lvert \xi_{l} \rvert
= a \geq C_{\ast} M
\geq C_{\ast} \lVert q_{\ell} \rVert_{H^{s} ( \Omega )}.
\end{math}
From Lemma~\ref{lemma:CGOsolution}, there exist CGO solutions
\[
u_{l} (x) = e^{i \zeta_{l} x} \bigl( 1 + \psi_{l} (x) \bigr)
\]
to equation (\ref{eq:1}) with
$q = q_{l}$, where $\psi_{l}$ satisfies
\[
\lVert \psi_{l} \rVert_{H^{s} ( \Omega )}
\leq \frac{C}{\lvert \xi_{l} \rvert}
\lVert q_{l} \rVert_{H^{s} ( \Omega )} .
\]
Note that $\psi_{l}$ also satisfies the estimate
\begin{equation}\label{eq:estimateofpsi}
\lVert \psi_{l} \rVert_{H^{s} ( \Omega )}
\leq \frac{C}{\lvert \xi_{l} \rvert}
\lVert q_{l} \rVert_{H^{s} ( \Omega )}
\leq \frac{C M}{a} < \frac{C M}{C_{\ast} M} = \frac{C}{C_{\ast}} .
\end{equation}
Now, by Proposition~\ref{proposition:identityestimate} and using the relation $- r \omega = \zeta_{1} + \zeta_{2}$, we have that
\begin{align*}
& \left\lvert
 \int_{\Omega}
  \widetilde{q} (x) e^{- i r \omega \cdot x}
  ( 1 + \psi_{1} ) ( 1 + \psi_{2} ) \,
 d x
\right\rvert
= \left\lvert
 \int_{\Omega} ( q_{2} - q_{1} ) u_{1} u_{2} \, d x
\right\rvert \\
& \leq \left\lVert
 \left( u_{1} , \frac{\partial u_{1}}{\partial \nu} \right)
\right\rVert_{H^{1/2} \oplus H^{-1/2}}
\left\lVert
 \left( u_{2} , \frac{\partial u_{2}}{\partial \nu} \right)
\right\rVert_{H^{1/2} \oplus H^{-1/2}}
\dist ( \mathcal{C}_{q_{1}} , \mathcal{C}_{q_{1}} ).
\end{align*}
Subsequently, we obtain
\begin{align}
\label{eq:Fq01}
& \lvert \mathcal{F} \widetilde{q} ( r \omega ) \rvert
= \left\lvert
 \int_{\Omega} \widetilde{q} (x) e^{- i r \omega \cdot x} \, d x
\right\rvert \\
& \leq \left\lvert
 \int_{\Omega}
  \widetilde{q} (x) e^{- i r \omega \cdot x}
  ( 1 + \psi_{1} ) ( 1 + \psi_{2} ) \,
 d x
\right\rvert \notag \\
& \hspace*{27ex} \mbox{}
+ \left\lvert
 \int_{\Omega}
  \widetilde{q} (x) e^{- i r \omega \cdot x}
  ( \psi_{1} + \psi_{2} + \psi_{1} \psi_{2} ) \,
 d x
\right\rvert \notag \\
& \leq \left\lVert
 \left( u_{1} , \frac{\partial u_{1}}{\partial \nu} \right)
\right\rVert_{H^{1/2} \oplus H^{-1/2}}
\left\lVert
 \left( u_{2} , \frac{\partial u_{2}}{\partial \nu} \right)
\right\rVert_{H^{1/2} \oplus H^{-1/2}}
\dist ( \mathcal{C}_{q_{1}} , \mathcal{C}_{q_{1}} ) \notag \\
& \hspace*{27ex} \mbox{}
+ \left\lvert
 \int_{\Omega}
  \widetilde{q} (x) e^{- i r \omega \cdot x}
  ( \psi_{1} + \psi_{2} + \psi_{1} \psi_{2} ) \,
 d x
\right\rvert . \notag
\end{align}

In view of \eqref{eq:Fq01}, we want to estimate
\begin{math}
\bigl\lVert
 ( u_{l} , \partial u_{l} / \partial \nu )
\bigr\rVert_{H^{1/2} \oplus H^{-1/2}}
\end{math}. Recall that $u_{l}$ solves (\ref{eq:1}) with $q = q_{l}$. Using assumptions
$\lVert q_{l} \rVert_{H^{s} ( \Omega)} \leq M$, and
$s > n/2$, and $k \geq 1$, we have that
\[
\left\lVert
 \frac{\partial u_{l}}{\partial \nu}
\right\rVert_{H^{-1/2} ( \partial \Omega )}
\leq C k^{2} \lVert u_{l} \rVert_{L^{2} ( \Omega )}
+ C \lVert \nabla u_{l} \rVert_{L^{2} ( \Omega )}
\]
and thus
\[
\left\lVert
 \left(
  u_{l} , \frac{\partial u_{l}}{\partial \nu}
 \right)
\right\rVert_{H^{1/2} \oplus H^{-1/2}}
\leq C k^{2} \lVert u_{l} \rVert_{L^{2} ( \Omega )}
+ C \lVert \nabla u_{l} \rVert_{L^{2} ( \Omega )}.
\]
We now choose $R_{0} > 0$ large enough such that
$\Omega \subset B_{R_{0}} (0)$. Then we have
\[
\lvert u_{l} (x) \rvert
\leq e^{- \xi_{l} \cdot x}
\bigl( 1 + \lvert \psi_{l} (x) \rvert \bigr)
\leq C e^{\lvert \xi_{l} \rvert R_{0}} = C e^{a R_{0}}
\]
since
\[
\lvert \psi_{l} (x) \rvert
\leq \lVert \psi_{l} \rVert_{L^{\infty} ( \Omega )}
\leq C \lVert \psi_{l} \rVert_{H^{s} ( \Omega )}
\leq C
\]
by $s > n/2$ and (\ref{eq:estimateofpsi}).  It follows that
\[
\lVert u_{l} \rVert_{L^{2} ( \Omega )} \leq C e^{a R_{0}} .
\]
On the other hand, in view of \begin{math}
\lVert \nabla \psi_{l} \rVert_{L^{2} ( \Omega )}
\leq \lVert \psi_{l} \rVert_{H^{s} ( \Omega )}
\leq C
\end{math}
($s > n/2 \geq 3/2 > 1$) and (\ref{eq:estimateofpsi}), we can estimate
\begin{align*}
\lVert \nabla u_{l} \rVert_{L^{2} ( \Omega )}
& = \left\lVert
 i u_{l} \zeta_{l}
 + e^{i \zeta_{l} \cdot \bullet} \nabla \psi_{l}
\right\rVert_{L^{2} ( \Omega )}
\leq \lvert \zeta_{l} \rvert \lVert u_{l} \rVert_{L^{2} ( \Omega )}
+ e^{\lvert \xi_{l} \rvert R_{0}}
\lVert \nabla \psi_{l} \rVert_{L^{2} ( \Omega )} \\
& \leq C \bigl( k + \lvert \xi_{l} \rvert \bigr) e^{a R_{0}}
+ C e^{\lvert \xi_{l} \rvert R_{0}}
= C ( k + a ) e^{a R_{0}} + C e^{a R_{0}} \leq C k e^{C a}.
\end{align*}
Summing up, we obtain
\begin{align}
\label{eq:ududnu}
\left\lVert
 \left(
  u_{l} , \frac{\partial u_{l}}{\partial \nu}
 \right)
\right\rVert_{H^{1/2} \oplus H^{-1/2}}
& \leq C k^{2} \lVert u_{l} \rVert_{L^{2} ( \Omega )}
+ C \lVert \nabla u_{l} \rVert_{L^{2} ( \Omega )} \\
& \leq C k^{2} e^{C a} + C k e^{C a}
\leq C k^{2} e^{C a}. \notag
\end{align}
Note that here $C$ depends on $n$, $s$, $M$, and the diameter of $\Omega$.

Let  $\chi \in C_{0}^{\infty} ( \Omega )$ be a cut-off function satisfying $\chi \equiv 1$ near $\supp ( q_{1} - q_{2} )$,  then  we have
\begin{align}
\label{eq:FqII}
& \left\lvert
 \int_{\Omega}
  \widetilde{q} (x) e^{- i r \omega \cdot x}
  ( \psi_{1} + \psi_{2} + \psi_{1} \psi_{2} ) \,
 d x
\right\rvert \\
& \hspace*{5ex} = \left\lvert
 \int_{\Omega}
  \widetilde{q} (x) \chi (x) e^{- i r \omega \cdot x}
  ( \psi_{1} + \psi_{2} + \psi_{1} \psi_{2} ) \,
 d x
\right\rvert \notag \\
& \hspace*{5ex} \leq \int_{\Omega}
 \lvert \widetilde{q} (x) \rvert
 \lvert
  \chi ( \psi_{1} + \psi_{2} + \psi_{1} \psi_{2} )
 \rvert \,
d x \notag \\
& \hspace*{5ex}
\leq \lVert \widetilde{q} \rVert_{H^{-s} ( \Omega )}
\lVert
 \chi ( \psi_{1} + \psi_{2} + \psi_{1} \psi_{2} )
\rVert_{H^{s} ( \Omega )} . \notag
\end{align}
Since $s > n/2$ and (\ref{eq:estimateofpsi}), we can estimate
\begin{align}
\label{eq:FqIIchipsi}
& \lVert
 \chi ( \psi_{1} + \psi_{2} + \psi_{1} \psi_{2} )
\rVert_{H^{s} ( \Omega )} \\
& \leq \lVert \chi \rVert_{H^{s} ( \Omega )}
\bigl(
 \lVert \psi_{1} \rVert_{H^{s} ( \Omega )}
 + \lVert \psi_{2} \rVert_{H^{s} ( \Omega )}
 + \lVert \psi_{1} \rVert_{H^{s} ( \Omega )}
 \lVert \psi_{2} \rVert_{H^{s} ( \Omega )}
\bigr) \notag \\
& \leq \lVert \chi \rVert_{H^{s} ( \Omega )}
\left(
 \frac{C M}{a} + \frac{C M}{a}
 + \frac{C}{C_{\ast}} \cdot \frac{C M}{a}
\right)
\leq \frac{C}{a} \notag .
\end{align}
Finally, (\ref{eq:lemma01}) follows from (\ref{eq:Fq01}),
(\ref{eq:ududnu}), (\ref{eq:FqII}), and (\ref{eq:FqIIchipsi}).
\end{proof}
The following lemma is an easy corollary of Lemma~\ref{lemma:keyestimate}.
\begin{lemma}\label{lemma:keyestimate2}
Suppose that the assumptions in Theorem~\ref{theorem:main} hold.
Let $R > C_{\ast} M$ with $C_{\ast}$ being the constant given in Lemma~\ref{lemma:CGOsolution}.
Then for $k \geq 1$, $r \geq 0$ and
$\omega \in \mathbb{R}^{n}$
with $\lvert \omega \rvert = 1$, the following estimates hold true:
if $0 \leq r \leq k + R$ then
\begin{equation}\label{eq:estimateforsmallr}
\lvert \mathcal{F} \widetilde{q} ( r \omega ) \rvert
\leq C k^{4} e^{C R}
\dist ( \mathcal{C}_{q_{1}} , \mathcal{C}_{q_{2}} )
+ \frac{C}{R} \lVert \widetilde{q} \rVert_{H^{-s} ( \mathbb{R}^{n} )} ;
\end{equation}
if $r \geq k + R$ then
\begin{equation}\label{eq:estimateforlarger}
\lvert \mathcal{F} \widetilde{q} ( r \omega ) \rvert
\leq C k^{4} e^{C r}
\dist ( \mathcal{C}_{q_{1}} , \mathcal{C}_{q_{2}} )
+ \frac{C}{r} \lVert \widetilde{q} \rVert_{H^{-s} ( \mathbb{R}^{n} )} .
\end{equation}
\begin{proof}
It is enough to take $a = R$ when $0 \leq r \leq k + R$,
and take $a = r$ when $r \geq k + R$ in Lemma~\ref{lemma:keyestimate}.
\end{proof}
\end{lemma}
Now we prove our main theorem.
\begin{proof}[Proof of Theorem~\ref{theorem:main}]
Written in polar coordinates, we have that
\begin{align}
\label{eq:qHminuss}
\lVert \widetilde{q} \rVert_{H^{-s} ( \mathbb{R}^{n} )}^{2}
& = C \int_{0}^{\infty} \int_{\lvert \omega \rvert = 1}
 \lvert \mathcal{F} \widetilde{q} ( r \omega ) \rvert^{2}
 ( 1 + r^{2} )^{-s} r^{n-1} \,
d \omega d r \\
& = C \biggl(
 \int_{0}^{k + R} \int_{\lvert \omega \rvert = 1}
  \lvert \mathcal{F} \widetilde{q} ( r \omega ) \rvert^{2}
  ( 1 + r^{2} )^{-s} r^{n-1} \,
 d \omega d r \notag \\
 & \hspace*{7ex} \mbox{}
 + \int_{k + R}^{T} \int_{\lvert \omega \rvert = 1}
  \lvert \mathcal{F} \widetilde{q} ( r \omega ) \rvert^{2}
  ( 1 + r^{2} )^{-s} r^{n-1} \,
 d \omega d r \notag \\
 & \hspace*{7ex} \mbox{}
 + \int_{T}^{\infty} \int_{\lvert \omega \rvert = 1}
  \lvert \mathcal{F} \widetilde{q} ( r \omega ) \rvert^{2}
  ( 1 + r^{2} )^{-s} r^{n-1} \,
 d \omega d r
\biggr) \notag \\
& =: C ( I_{1} + I_{2} + I_{3} ), \notag
\end{align}
where $R > C_{\ast} M$ and $T \geq k + R$
are parameters which will be chosen later.

Our task now is to estimate each integral separately.
We begin with $I_{3}$. Since
\begin{math}
\lvert \mathcal{F} \widetilde{q} ( r \omega ) \rvert
\leq C \lVert q_{1} - q_{2} \rVert_{L^{2} ( \Omega )}
\end{math},
$q_{1} - q_{2} \in H_{0}^{s} ( \Omega )$ and $s > n/2$,
we get
\begin{align}
\label{eq:I3}
I_{3}
& \leq C \int_{T}^{\infty}
 \lVert q_{1} - q_{2} \rVert_{L^{2} ( \Omega )}^{2}
 ( 1 + r^{2} )^{-s} r^{n-1} \,
d r
\leq C T^{- m}
\lVert q_{1} - q_{2} \rVert_{L^{2} ( \Omega )}^{2} \\
& \leq C T^{- m}
\left(
 \varepsilon \lVert q_{1} - q_{2} \rVert_{H^{-s} ( \Omega )}^{2}
 + \frac{1}{\varepsilon}
 \lVert q_{1} - q_{2} \rVert_{H^{s} ( \Omega )}^{2}
\right) \notag\\
& \leq C T^{- m}
\left(
 \varepsilon \lVert \widetilde{q} \rVert_{H^{-s} ( \mathbb{R}^{n} )}^{2}
 + \frac{1}{\varepsilon}
\right) \notag
\end{align}
for $\varepsilon > 0$, where $m := 2 s - n$.

On the other hand, by estimate~(\ref{eq:estimateforsmallr}), we can obtain
\begin{align}
\label{eq:I1}
I_{1}
& \leq \int_{0}^{k + R}
 \left(
  C k^{4} e^{C R}
  \dist ( \mathcal{C}_{q_{1}} , \mathcal{C}_{q_{2}} )
  + \frac{C}{R} \lVert \widetilde{q} \rVert_{H^{-s} ( \mathbb{R}^{n} )}
 \right)^{2}
 ( 1 + r^{2} )^{-s} r^{n-1} \,
d r \\
& \leq C \left(
 k^{8} e^{C R}
 \dist ( \mathcal{C}_{q_{1}} , \mathcal{C}_{q_{2}} )^{2}
 + \frac{1}{R^{2}}
 \lVert \widetilde{q} \rVert_{H^{-s} ( \mathbb{R}^{n} )}^{2}
\right)
\int_{0}^{\infty} ( 1 + r^{2} )^{-s} r^{n-1} \, d r \notag \\
& = C \left(
 k^{8} e^{C R}
 \dist ( \mathcal{C}_{q_{1}} , \mathcal{C}_{q_{2}} )^{2}
 + \frac{1}{R^{2}}
 \lVert \widetilde{q} \rVert_{H^{-s} ( \mathbb{R}^{n} )}^{2}
\right) . \notag
\end{align}
In the same way, using estimate~(\ref{eq:estimateforlarger}), we have
\begin{align}
\label{eq:I2}
I_{2}
& \leq C \int_{k+R}^{T}
 \left(
  C k^{4} e^{C r}
  \dist ( \mathcal{C}_{q_{1}} , \mathcal{C}_{q_{2}} )
  + \frac{C}{r} \lVert \widetilde{q} \rVert_{H^{-s} ( \mathbb{R}^{n} )}
 \right)^{2}
 ( 1 + r^{2} )^{-s} r^{n-1} \,
d r \\
& \leq C k^{8}
\dist ( \mathcal{C}_{q_{1}} , \mathcal{C}_{q_{2}} )^{2}
\int_{k+R}^{T} e^{C r}  ( 1 + r^{2} )^{-s} r^{n-1} \, d r \notag \\
& \hspace*{20ex} \mbox{}
+ C \lVert \widetilde{q} \rVert_{H^{-s} ( \mathbb{R}^{n} )}^{2}
\int_{k+R}^{T}  ( 1 + r^{2} )^{-s} r^{n-1} \, d r \notag \\
& \leq C \left(
 k^{8} e^{C T}
 \dist ( \mathcal{C}_{q_{1}} , \mathcal{C}_{q_{2}} )^{2}
 + \frac{1}{R^{2}}
 \lVert \widetilde{q} \rVert_{H^{-s} ( \mathbb{R}^{n} )}^{2}
\right) , \notag
\end{align}
where we have used
\begin{align*}
& \int_{k+R}^{T} e^{C r} ( 1 + r^{2} )^{-s} r^{n-1} \, d r
\leq e^{C T} \int_{k+R}^{T} ( 1 + r^{2} )^{-s} r^{n-1} \, d r \\
& \hspace*{15ex}
\leq e^{C t} \int_{0}^{\infty} ( 1 + r^{2} )^{-s} r^{n-1} \, d r
= C e^{C T},
\end{align*}
\begin{align*}
\int_{k+R}^{T}  ( 1 + r^{2} )^{-s} r^{n-1} \, d r
& \leq \int_{k+R}^{T} r^{- 2 s + n-1} \, d r \\
& \leq \frac{1}{2 s - n + 2}
\frac{1}{( k + R )^{2 s - n + 2}}
\leq \frac{C}{( k + R )^{2}}
\leq \frac{C}{R^{2}},
\end{align*}
and $s > n / 2$, $k \geq 1$.
Combining (\ref{eq:qHminuss})--(\ref{eq:I2}) gives
\begin{align}
\label{eq:tildeqRepsilon}
\lVert \widetilde{q} \rVert_{H^{-s} ( \mathbb{R}^{n} )}^{2}
& \leq C ( I_{1} + I_{2} + I_{3} ) \\
& \leq C \left(
 k^{8} e^{C R}
 \dist ( \mathcal{C}_{q_{1}} , \mathcal{C}_{q_{2}} )^{2}
 + \frac{1}{R^{2}}
 \lVert \widetilde{q} \rVert_{H^{-s} ( \mathbb{R}^{n} )}^{2}
\right) \notag \\
& \hspace*{4ex} \mbox{}
+ C \left(
 k^{8} e^{C T}
 \dist ( \mathcal{C}_{q_{1}} , \mathcal{C}_{q_{2}} )^{2}
 + \frac{1}{R^{2}}
 \lVert \widetilde{q} \rVert_{H^{-s} ( \mathbb{R}^{n} )}^{2}
\right) \notag \\
& \hspace*{4ex} \mbox{}
+ C T^{- m}
\left(
 \varepsilon \lVert \widetilde{q} \rVert_{H^{-s} ( \mathbb{R}^{n} )}^{2}
 + \frac{1}{\varepsilon}
\right) \notag \\
& \leq C \left( \frac{2}{R^{2}} + \varepsilon T^{-m} \right)
\lVert \widetilde{q} \rVert_{H^{-s} ( \mathbb{R}^{n} )}^{2}
+ C k^{8} e^{C R}
\dist ( \mathcal{C}_{q_{1}} , \mathcal{C}_{q_{2}} )^{2} \notag \\
& \hspace*{5ex} \mbox{} + C k^{8} e^{C T}
\dist ( \mathcal{C}_{q_{1}} , \mathcal{C}_{q_{2}} )^{2}
+ \frac{C T^{-m}}{\varepsilon} . \notag
\end{align}

To continue, we consider the following two cases:
\[
\mbox{(i) } k + R \leq p \log \frac{1}{A} \quad
\mbox{and} \quad
\mbox{(ii) } k + R \geq p \log \frac{1}{A} ,
\]
where $R > C_{\ast} M$ and
$p > 0$ are constants which will be determined later.  We begin with the first case (i).
Taking
\begin{equation}\label{eq:iR}
R > 2 \sqrt{C}
\end{equation}
and $\varepsilon = c T^{m}$
($c \ll 1$), we deduce that
\begin{equation}\label{eq:tildeqPhiT}
\lVert \widetilde{q} \rVert_{H^{-s} ( \mathbb{R}^{n} )}^{2}
\leq C k^{8} A
+ C k^{8} e^{C T} A
+ C T^{- 2 m}
\end{equation}
for any $T \geq k + R$ by (\ref{eq:tildeqRepsilon}),
where
$A = \dist ( \mathcal{C}_{q_{1}} , \mathcal{C}_{q_{2}} )^{2}$.

Now we choose $T = p \log (1/A)$,
which is greater than or equal to $k + R$ by the condition (i).
Our current aim is to show that there exists
$C_{1} > 0$ such that
\begin{equation}\label{eq:iaim1}
k^{8} e^{C T} A
\leq C_{1}
\left( k + \log \frac{1}{A} \right)^{- 2 m}
\end{equation}
and
\begin{equation}\label{eq:iaim2}
T^{- 2 m}
\leq C_{1}
\left( k + \log \frac{1}{A} \right)^{- 2 m} .
\end{equation}
Substituting (\ref{eq:iaim1}) and (\ref{eq:iaim2})
into (\ref{eq:tildeqPhiT})
clearly implies (\ref{eq:main}).
We remark that  (\ref{eq:iaim2}) is equivalent to
\begin{equation}\label{eq:iaim2a}
C_{1}^{- 1 / 2 m} \left( k + \log \frac{1}{A} \right)
\leq p \log \frac{1}{A} .
\end{equation}
Since we have
\[
k + \log \frac{1}{A} \leq (k+R) + \log \frac{1}{A}
\leq ( p + 1 ) \log \frac{1}{A}
\]
by (i), condition (\ref{eq:iaim2a})
(i.e.\ (\ref{eq:iaim2})) holds whenever
\begin{equation}\label{eq:C1pcond1}
C_{1}^{- 1 / 2 m} \leq \frac{p}{p + 1} .
\end{equation}
On the other hand, condition (\ref{eq:iaim1})
is equivalent to
\begin{equation}\label{eq:iaim1a}
8 \log k + ( C p - 1 ) \log \frac{1}{A}
+ 2 m \log \left( k + \log \frac{1}{A} \right)
\leq \log C_{1} .
\end{equation}
Using (i), we can bound the left-hand side of (\ref{eq:iaim1a}) by
\[
( \mbox{LHS of (\ref{eq:iaim1a})} )
\leq 8 \log p + 2 m \log (p+1)
+ ( C p - 1 ) \log \frac{1}{A} + 2 ( m + 4 ) \log \log \frac{1}{A}.
\]
Choosing
\begin{equation}\label{eq:C1pcond2}
p \leq \frac{1}{2 C},
\end{equation}
we can see that
\begin{align*}
& ( \mbox{LHS of (\ref{eq:iaim1a})} ) \\
& \leq 8 \log \frac{1}{2 C}
+ 2 m \log \left( \frac{1}{2 C} + 1 \right)
- \frac{1}{2} \log \frac{1}{A}
+ 2 ( m + 4 ) \log \log \frac{1}{A} \\
& \leq 8 \log \frac{1}{2 C}
+ 2 m \log \left( \frac{1}{2 C} + 1 \right)
+ \max_{z \geq 2}
\left( - \frac{1}{2} z + 2 ( m + 4 ) \log z \right) \\
& = 8 \log \frac{1}{2 C}
+ 2 m \log \left( \frac{1}{2 C} + 1 \right)
+ 2 ( m + 4 ) \bigl( \log ( 4 m + 16 ) - 1 \bigr) .
\end{align*}
Therefore, condition (\ref{eq:iaim1a})
(i.e.\ (\ref{eq:iaim1})) is satisfied provided
\begin{equation}\label{eq:C1pcond3}
8 \log \frac{1}{2 C}
+ 2 m \log \left( \frac{1}{2 C} + 1 \right)
+ 2 ( m + 4 ) \bigl( \log ( 4 m + 16 ) - 1 \bigr)
\leq \log C_{1} .
\end{equation}

Next we consider case (ii).
We choose $T = k + R$ and observe that the term $I_{2}$ in (\ref{eq:qHminuss}) does not
appear in this case. Hence, instead of
(\ref{eq:tildeqRepsilon}), we have
\begin{align*}
& \lVert \widetilde{q} \rVert_{H^{-s} ( \mathbb{R}^{n} )}^{2} \\
& \leq C \left( \frac{1}{R^{2}} + \varepsilon T^{-m} \right)
\lVert \widetilde{q} \rVert_{H^{-s} ( \mathbb{R}^{n} )}^{2}
+ C k^{8} e^{C R}
\dist ( \mathcal{C}_{q_{1}} , \mathcal{C}_{q_{2}} )^{2}
+ \frac{C T^{-m}}{\varepsilon}
\end{align*}
Setting $\varepsilon = T^{m} / R^{2}$ implies that
\[
\lVert \widetilde{q} \rVert_{H^{-s} ( \mathbb{R}^{n} )}^{2}
\leq \frac{2 C}{R^{2}}
\lVert \widetilde{q} \rVert_{H^{-s} ( \mathbb{R}^{n} )}^{2}
+ C k^{8} e^{C R} A
+ C R^{2} ( k + R )^{- 2 m} .
\]
Now we choose
\begin{equation}\label{eq:iiR}
R > 2 \sqrt{C}
\end{equation}
and obtain that
\[
\lVert \widetilde{q} \rVert_{H^{-s} ( \mathbb{R}^{n} )}^{2}
\leq C k^{8} A
+ C ( k + R )^{- 2 m} ,
\]
which implies the desired estimate (\ref{eq:main})
since from condition (ii) we have
\[
k + R \geq \frac{k}{2} + \frac{k + R}{2}
\geq \frac{k}{2} + \frac{p}{2} \log \frac{1}{A}
\geq \frac{\min \{ p , 1 \}}{2}
\left( k + \log \frac{1}{A} \right) .
\]

As the last step, we choose appropriate $R, p$, and $C_{1}$ to complete the proof.
We first pick $R > C_{\ast} M$ sufficiently large satisfying
(\ref{eq:iR}) and (\ref{eq:iiR}) and then choose $p$ small enough satisfying (\ref{eq:C1pcond2}). Finally, we
take $C_{1}$ large enough satisfying
(\ref{eq:C1pcond1}) and (\ref{eq:C1pcond3}).
\end{proof}

\section{Conclusion}

We think that increasing stability is an important feature of the inverse boundary problem for the Schr\"odinger potential which should lead to higher resolution of numerical algorithms. It is important to collect numerical evidence of this phenomenon. Our method is based on the CGO solutions constructed in \cite{ha} where the constants in Lemma 2.1 are explicit. So most likely one can give explicit constants in
Theorem 1.1 at least for particular domains $\Omega$ like balls. Contrary to the acoustic case \cite{NUW}, the constants in the estimate \eqref{eq:main} depend only polynomially on $k$. It is an important and challenging question to determine whether the exponential dependence on $k$ of the estimates in \cite{NUW} is indeed generic if there are no assumptions on rays.

\end{document}